\documentclass[12pt]{amsart}
\usepackage{amsmath,amssymb,amsthm,amscd}
\newtheorem{thm}{Theorem}

\newtheorem{conjec}[thm]{Conjecture}
\newtheorem{lem}[thm]{Lemma}

\newtheorem*{rem}{Remark}

\newcommand{\F}{\mathbb{F}}
\newcommand{\Q}{\mathbb{Q}}
\newcommand{\Z}{\mathbb{Z}}
\newcommand{\G}{\mathbb{G}}
\newcommand{\C}{\mathbb{C}}

\newcommand{\Gal}{{\rm Gal}}
\newcommand{\legen}[2]{\genfrac{(}{)}{}{}{#1}{#2}}
\newcommand{\artin}[2]{\genfrac{(}{)}{}{}{#1}{#2}}
\newcommand{\ord}{{\rm ord}}
\newcommand{\abs}[1]{\left\lvert#1\right\rvert}
\begin{document}

\title{Divisibility properties of the Fibonacci entry point}
\author{Paul Cubre}
\address{Department of Mathematics, Penn State University, State College, PA 16802}
\email{pcubre@gmail.com}
\author{Jeremy Rouse}
\address{Department of Mathematics, Wake Forest University, Winston-Salem, NC 27109}
\email{rouseja@wfu.edu}
\thanks{The first author was partially supported by the Wake Forest University
Graduate School. The second author was supported by NSF grant DMS-0901090}
\begin{abstract}
For a prime $p$, let $Z(p)$ be the smallest positive integer $n$ so that
$p$ divides $F_{n}$, the $n$th term in the Fibonacci sequence. Paul Bruckman
and Peter Anderson conjectured a formula for $\zeta(m)$, the density of primes
$p$ for which $m | Z(p)$ on the basis of numerical evidence. We prove
Bruckman and Anderson's conjecture by studying the algebraic group
$G : x^{2} - 5y^{2} = 1$ and relating $Z(p)$ to the order of $\alpha =
(3/2,1/2) \in G(\F_{p})$. We are then able to use Galois theory and the
Chebotarev density theorem to compute $\zeta(m)$. 
\end{abstract}

\maketitle

\section{Introduction and Statement of Results}

Let $F_{n}$ denote the Fibonacci sequence defined as usual by $F_{0} =
0$, $F_{1} = 1$ and $F_{n} = F_{n-1} + F_{n-2}$ for $n \geq 2$. If $p$
is a prime number, the smallest positive integer $m$ for which $p |
F_{m}$ is called the Fibonacci entry of $p$, $Z(p)$. For example, we
have $Z(11) = 10$ since $11 | F_{10} = 55$ is the smallest Fibonacci
number that is a multiple of $11$.

It is well-known that for every prime $p$, $Z(p) \leq p+1$ (in fact, a
proof of this follows from Lemma~\ref{grouporder} and
Lemma~\ref{Zplem} in Section~\ref{modp}). In 1913, Carmichael (see
\cite{Carmichael}, Theorem XXI) proved that if $m \ne 1, 2, 6, 12$,
then there is a prime number $p$ so that $Z(p) = m$. It is not presently
known if there are infinitely many primes $p$ for which $Z(p) =
p+1$. 

The main question we study is, given a positive integer $m$,
how often does $m$ divide $Z(p)$? A natural conjecture would be
that $Z(p)$ is ``random'' mod $m$ and so the answer should be $1/m$.
However, Lagarias proved in 1985 (see \cite{Lagarias} and \cite{LagariasCor}) 
that the density of primes $p$ so that $Z(p)$ is even is $2/3$. More precisely,
\[
  \lim_{x \to \infty}
  \frac{\# \{ p \leq x : p \text{ is prime and } Z(p) \text{ is even} \}}{\pi(x)}  = \frac{2}{3},
\]
where $\pi(x)$ is the number of primes $\leq x$. Motivated by this work,
Bruckman and Anderson in \cite{BA} gathered numerical data and conjectured
a formula for
\[
  \zeta(m) := \lim_{x \to \infty} M(m,x) / \pi(x),
\]
where $M(m,x) = \# \{ p \leq x : p \text{ is prime and } m | Z(p) \}$. Their
conjecture is the following. 

\begin{conjec}[Conjecture 3.1 of \cite{BA}]
\label{BAconj}
If $m = q^{e}$ is a prime power (with $e \geq 1$), then
\[
  \zeta(q^{e}) = \frac{q^{2-e}}{q^{2} - 1}.
\]
For an arbitrary positive integer $m$, we have
\[
  \zeta(m) = \rho(m) \prod_{q^{j} \| m} \zeta(q^{j}).
\]
where the product is over all prime powers occurring in the
prime factorization of $m$, and
\[
  \rho(m) = \begin{cases} 
    1 & \text{ if } 10 \nmid m\\
    \frac{5}{4} & \text{ if } m \equiv 10 \pmod{20}\\
    \frac{1}{2} & \text{ if } 20 | m.
\end{cases}
\]
\end{conjec}

The main result of the current paper is a proof of this conjecture.

\begin{thm}
\label{main}
Conjecture~\ref{BAconj} is true for every positive integer $m$.
\end{thm}

One consequence of this result is that the numbers $Z(p)$ show a bias
toward being composite. For example, if $q$ is prime, then $\zeta(q) =
\frac{q}{q^{2} - 1} \approx \frac{1}{q} + \frac{1}{q^{2}}$, and so the
numbers $Z(p)$ are divisible by $q$ more often than entries in a
random sequence.

To prove Theorem~\ref{main}, we study the algebraic group
\[
  G : x^{2} - 5y^{2} = 1.
\]
This is a twisted torus isomorphic to $\mathbb{G}_{m}$ over $\Q(\sqrt{5})$.
The group law is given by $(x_{1},y_{1}) * (x_{2},y_{2}) = 
(x_{1} x_{2} + 5 y_{1} y_{2}, x_{1} y_{2} + x_{2} y_{1})$. We consider the
point $\alpha = (3/2,1/2) \in G(\Q)$, and we show (see Lemma~\ref{multiples}) 
that
\[
  n \alpha = \overbrace{\alpha + \alpha + \cdots + \alpha}^{n~\text{times}}
  = (L_{2n}/2, F_{2n}/2).
\]
Here $L_{n}$ is the $n$th Lucas number. These are defined by $L_{0} =
2$, $L_{1} = 1$ and $L_{n} = L_{n-1} + L_{n-2}$ for $n \geq 2$. In
Lemma~\ref{Zplem}, we use this to relate the Fibonacci entry point of
$p$ to the order of $\alpha \in G(\F_{p})$.

Next, we study the density of primes $p$ for which $m$ divides the
order of $\alpha \in G(\F_{p})$. In the context of untwisted tori,
questions of this nature are quite familiar.  They were first
considered by Hasse in \cite{Hasse1} and \cite{Hasse2}, and very general
results of this type are due to Ballot \cite{Ballot} and Moree \cite{Moree}. 
Twisted tori of the type mentioned above are considered by Jones and the
second author in \cite{JonesRouse}, and in this paper criteria are
given that will guarantee that if $\alpha \in G(\Q)$ is fixed and
$\ell$ is a prime number, then the density of primes $p$ for which
$\ell$ divides the order of $\alpha \in G(\F_{p})$ exists and equals
$\frac{\ell}{\ell^{2} - 1}$.

In the present paper, we extend these results to arbitrary positive
integers $m$. To make the extension to prime powers, we consider
preimages of $\alpha \in G(\F_{p})$ under multiplication by $\ell$. We
say that $\alpha$ has an $\ell^{n}$th preimage if there is a $\beta
\in G(\F_{p})$ so that $\ell^{n} \beta = \alpha$. In
Lemma~\ref{orderpreimage}, we relate the order of $\alpha$ with the
number of preimages of $\alpha$ under the multiplication by $\ell$.
Questions of this type are common in arithmetic dynamics, and these
have connections with the Diffie-Hellman key exchange protocol (see
\cite{DiffieHellman}). In this context, a generator $\alpha \in
\F_{q}^{\times}$ is chosen and an integer $x$ that encodes a message
is also chosen. Computing $y = \alpha^{x}$ given $\alpha$ and $x$ is
straightforward, but computing $x$ given $\alpha$ and $y$ is quite
difficult (note that $\alpha$ is a preimage of $y$).

To study how often $\alpha$ has an $\ell^{n}$th preimage, we show that there
are $\ell^{n}$ elements $P_{n,r}$ of $G(\C)$ so that $\ell^{n} P_{n,r}
= \alpha$. We let $K_{\ell^{k}}$ be the field obtained by adjoining
all of the $x$ and $y$-coordinates of the $\ell^{k}$th preimages to $\Q$. Then,
we essentially show that there is an $n$th preimage of $\alpha$ in $G(\F_{p})$
if and only if the Artin symbol $\artin{K_{\ell^{k}}/\Q}{\mathfrak{p}}$
fixes an $n$th preimage for all prime ideals $\mathfrak{p}$ above $p$.
Finally, we compute $\Gal(K_{\ell^{k}}/\Q)$ and compute the relevant
density. As a consequence, we are able to prove another conjecture
of Anderson and Bruckman.

\begin{thm}[Conjecture 2.1 of \cite{BA}]
\label{BAconj2}
Let $p$ be a prime and $\epsilon_{p} = \legen{p}{5}$. Given a prime
$q$ and integers $x$, $i$ and $j$ with $i \geq j \geq 0$, let
$M(q,x,i,j)$ denote the number of primes $p \leq x$ such that
$q^{i} \| (p - \epsilon_{p})$ and $q^{j} \| Z(p)$. Then
\[
  \zeta(q;i,j) := \lim_{x \to \infty}
  \frac{M(q,x,i,j)}{\pi(x)} = 
  \begin{cases}
    \frac{q-2}{q-1} & \text{ if } i = j = 0,\\
    q^{-2i} & \text{ if } i \geq 1 \text{ and } j = 0,\\
    \frac{q-1}{q^{2i-j+1}} & \text{ otherwise.}
  \end{cases}
\]
\end{thm}

To handle the case that $m = \prod_{i=1}^{r} \ell_{i}^{s_{i}}$ is composite,
we must show that the fields $K_{\ell_{i}^{s_{i}}}$, $1 \leq i \leq r$
are (almost) linearly disjoint. The complication in the formula for
$\zeta(m)$ arises from the fact that $\Q(\sqrt{5}) \subseteq
K_{2^{a}} \cap K_{5^{b}}$.

An outline of the paper is as follows. We review the relevant algebraic
number theory in Section~\ref{back}. In Section~\ref{modp} we connect
the Fibonacci sequence with the arithmetic of $G(\F_{p})$ and $Z(p)$
with the order of $\alpha \in G(\F_{p})$, and in turn we connect
that with the number of preimages. In Section~\ref{galois} we define
a Galois representation and prove that it is surjective (except in the
case that $2$ and $5$ both divide $m$). Finally in Section~\ref{density}
we compute the relevant densities and prove Theorem~\ref{BAconj2}
and Theorem~\ref{main}.

\section{Background}
\label{back}

We begin by reviewing some algebraic number theory. For an
introduction to these ideas, see \cite{Mollin}. If $L/K$ is a Galois
extension of number fields and $\alpha \in L$, define the norm of
$\alpha$ to be
\[
  N_{L/K}(\alpha) = \prod_{\sigma \in \Gal(L/K)} \sigma(\alpha).
\]

If $K/\Q$ is a field
extension, let $\mathcal{O}_{K}$ be the ring of algebraic integers in
$K$. We say that a prime number $p$ ramifies in $K$ if in the
factorization
\[
  p \mathcal{O}_{K} = \prod_{i=1}^{n} \mathfrak{p}_{i}^{e_{i}}
\]
we have $e_{i} > 1$ for some $i$. For a fixed $K$, only finitely many
primes $p$ ramify, and those that ramify are precisely those that
divide the discriminant $\Delta_{K}$.

Suppose now that $K/\Q$ is a Galois extension and $\mathfrak{p}$ is a
prime ideal in $\mathcal{O}_{K}$. We say that $\mathfrak{p}$ is a prime
ideal above $p$ if $\mathfrak{p} \cap \Z = (p)$. This implies that
$\mathcal{O}_{K}/\mathfrak{p}$ is a finite extension of $\F_{p}$.
If $p$ is unramified in $K/\Q$, then there is a unique element 
$\sigma \in \Gal(K/\Q)$ for which
\[
  \sigma(\alpha) \equiv \alpha^{p} \pmod{\mathfrak{p}}
\]
for all $\alpha \in \mathcal{O}_{K}$. This element is called the
\emph{Artin symbol} of $\mathfrak{p}$ and is denoted
$\artin{K/\Q}{\mathfrak{p}}$. Let
\[
  \artin{K/\Q}{p} = \left\{ \artin{K/\Q}{\mathfrak{p}} : \mathfrak{p} \text{ is a prime ideal of } \mathcal{O}_{K} \text{ above } p \right\}.
\]
This set is a conjugacy class in $\Gal(K/\Q)$.

Let $\zeta_{n} = e^{2 \pi i / n}$. It follows from the definition of the
Artin symbol that if $p$ is a prime and $\mathfrak{p}$ is a prime
ideal above $p$ in $\mathcal{O}_{\Q(\zeta_{n})}$, then
$\artin{K/\Q}{\mathfrak{p}}(\zeta_{n}) = \zeta_{n}^{p}$.

We are now ready to state the Chebotarev density theorem. This result
is the key tool we will use to compute the densities mentioned in the
introduction. To state it, let $\pi(x)$ be the number of primes $\leq
x$.
\begin{thm}[\cite{IK}, page 143]
  Suppose that $K/\Q$ is a finite Galois extension. If $\mathcal{C}
  \subset \Gal(K/\Q)$ is a conjugacy class, then
\[
  \lim_{x \to \infty} \frac{\# \{ p \leq x : p \text{ is prime and } \artin{K/\Q}{p} = \mathcal{C} \}}{\pi(x)}
  = \frac{|\mathcal{C}|}{|\Gal(K/\Q)|}.
\]
\end{thm}

We will need a standard result in Galois theory. If $K_{1}$ and $K_{2}$
are two subfields of a field $E$, let $\langle K_{1}, K_{2} \rangle$
be the smallest subfield that contains both $K_{1}$ and $K_{2}$. 
The following result describes $\Gal(\langle K_{1}, K_{2} \rangle/F)$
in terms of $\Gal(K_{1}/F)$, $\Gal(K_{2}/F)$ and $\Gal((K_{1} \cap K_{2})/F)$.
\begin{thm}[\cite{Milne}, Proposition 3.20]
\label{linearlydisjoint}
  Let $K_{1}$ and $K_{2}$ be Galois extensions of a field $F$. Then
  $K_{1} \cap K_{2}$ and $\langle K_{1}, K_{2} \rangle$ are both
  Galois over $F$ and
\[
  \Gal(\langle K_{1}, K_{2} \rangle/F) \cong \{ (\sigma,\tau) : \sigma|_{K_{1} \cap K_{2}} = \tau|_{K_{1} \cap K_{2}} \} \subseteq \Gal(K_{1}/F) \times
  \Gal(K_{2}/F).
\]
In particular,
\[
  |\langle K_{1}, K_{2} \rangle : F| = \frac{|K_{1} : F| |K_{2} : F|}{|K_{1} \cap K_{2} : F|}.
\]
\end{thm}

Finally, given a group $G$, we will denote its order by $|G|$. If
$g \in G$ is an element, we will denote the order of $g$ by $|g|$. If
$n$ is a non-zero integer and $\ell$ is a prime number we will denote
by $\ord_{\ell}(n)$ the highest power of $\ell$ that divides $n$.

\section{Connection between the Fibonacci sequence and 
$G(\F_{p})$}
\label{modp}

In this section we will connect the order of the point $\alpha =
(3/2,1/2) \in G(\F_{p})$ with $Z(p)$, the smallest positive integer
$n$ so that $p | F_{n}$. Given a prime number $\ell$ we will also
relate $\ord_{\ell}(|\alpha|)$, $\ord_{\ell}(|G(\F_{p})|)$ and the
highest positive integer $m$ so that there is a $\beta \in G(\F_{p})$
with $\ell^{m} \beta = \alpha$.

If $K$ is a field, let
\[
  G(K) = \{ (x,y) \in K^{2} : x^{2} - 5y^{2} = 1 \}.
\]
The set $G(K)$ becomes an abelian group with the group operation
\[
  (x_{1}, y_{1}) * (x_{2}, y_{2}) = (x_{1} x_{2} + 5 y_{1} y_{2},
  x_{1} y_{2} + x_{2} y_{1}).
\]
The identity of $G(K)$ is $(1,0)$, and the inverse of $(x_{1},y_{1})$
is $(x_{1},-y_{1})$. The group $G$ is a twisted algebraic torus. It becomes
isomorphic to $\G_{m}$ over $K(\sqrt{5})$.

\begin{lem}
\label{twistiso}
Let $H(K) = \{ (x,y) \in K^{2} : xy = 1 \}$. If $\sqrt{5} \in K$ and
${\rm char}(K) \ne 2$, then
\[
  H(K) \cong G(K).
\]
\end{lem}
\begin{proof}
  Let $\phi: G(K) \rightarrow H(K)$ be given by
  $\phi(x,y)=(x+\sqrt{5}y,x-\sqrt{5}y)$. It is easily checked that $\phi$ is
a homomorphism. If $(x,y) \in H(K)$ then $\phi((x+y)/2,(x-y)/2\sqrt{5})=(x,y)$ 
and $((x+y)/2,(x-y)/2\sqrt{5})\in G(K)$. This shows that $\phi$ is
  surjective. Then $\ker \phi= \{(x,y)\in G(K): x+\sqrt{5}y=1 \text{
   and } x-\sqrt{5}y=1\}$. If ${\rm char}(K) \ne 2$, this implies that
$\ker \phi = \{ (1,0) \}$ and hence $\phi$ is injective. 
\end{proof}

A simple consequence of this is the following.

\begin{lem}
For any prime $p \ne 2, 5$, $G(\F_{p})$ is cyclic.
\end{lem}
\begin{proof}
  We must prove this for two cases: $p \equiv 1,4 \pmod{5}$ and $p
  \equiv 2,3 \pmod{5}$. In the first case, $5$ is a square in $\F_{p}$ and 
Lemma~\ref{twistiso} shows that $G(\F_{p}) \cong \F_{p}^{\times}$ which is cyclic.
In the second case, $5$ is a square in $\F_{p^{2}}$ and so
$G(\F_{p}) \subseteq G(\F_{p^{2}}) \cong \F_{p^{2}}^{\times}$. This shows that
$G(\F_{p})$ is a subgroup of a cyclic group and is hence cyclic.
\end{proof}

The following result gives a formula for the order of $G(\F_{p})$.

\begin{lem}
\label{grouporder}
Suppose that $p \ne 2,5$ is prime. We have
\[
  |G(\F_{p})| = \begin{cases}
    p-1 & \text{ if } p \equiv 1 \text{ or } 4 \pmod{5}\\
    p+1 & \text{ if } p \equiv 2 \text{ or } 3 \pmod{5}.
\end{cases}
\]
\end{lem}
\begin{proof}
  In the case that $p \equiv 1,4 \pmod{5}$ the previous lemma gives
  $|G(\F_p)|=|\F_p^\times|=p-1$.  However the following proof handles
  both cases. A line through the point $(1,0)$ intersects the curve
  $x^{2} - 5y^{2} = 1$ in another rational point if and only if the
  slope is rational; the same argument applies in $\F_{p}$.  Such a
  line has the form $y \equiv m(x-1) \pmod{p}$.  Every point in
  $G(\F_{p})$ other than $(1,0)$ lies on one such a line. Computing
  the intersection of this line with $x^{2} - 5y^{2} = 1$ shows that
  $x = -\frac{1 + 5m^{2}}{1-5m^{2}}$ and $y = -\frac{2m}{1-5m^{2}}$.
  This gives rise to a map $f : S \to G(\F_{p}) - \{ (1,0) \}$ where
  $S = \{ m \in \F_{p} : 5m^{2} \ne 1 \}$. The argument above shows
  that this map is surjective, and a straightforward calculation shows
  that $f$ is injective.  It follows that $|G(\F_{p}) = |S|$. We have that $|S| = p-1$ when $p \equiv 1, 4 \pmod{5}$ and $|S| = p+1$ otherwise. This completes the proof.
\end{proof}

The next result shows that the Fibonacci and Lucas sequences
occur as the coordinates of multiples of $\alpha = (3/2,1/2) \in G(\Q)$.

\begin{lem}
\label{multiples}
We have $n \alpha = \left(\frac{L_{2n}}{2}, \frac{F_{2n}}{2}\right)$.
\end{lem}
\begin{proof}
  We prove this by induction on $n$.  The base case for $n=1$ is
  $\alpha=(3/2,1/2)=(L_2/2,F_2/2)$.  Our induction hypothesis for $n$
  is $n\alpha=(L_{2n}/2,F_{2n}/2)$.  Then
  $n\alpha*\alpha=(L_{2n}/2,F_{2n}/2)*(3/2,1/2)=((3L_{2n}+5F_{2n})/4,(L_{2n}+3F_{2n})/4)$.
  Using the identity $5F_n=L_{n+1}+L_{n-1}$ on the first coordinate we
  get
\begin{align*}
3L_{2n}+5F_{2n}&=3L_{2n}+L_{2n+1}+L_{2n-1}\\
&=3L_{2n}+2L_{2n+1}-L_{2n}\\
&=2(L_{2n}+L_{2n+1}) =2L_{2n+2}.
\end{align*}
Then using the identity $L_n=F_{n+1}+F_{n-1}$ on the second coordinate we get
\begin{align*}
L_{2n}+3F_{2n}&=F_{2n+1}+F_{2n-1}+3F_{2n}\\
&=F_{2n+1}+F_{2n+1}-F_{2n}+3F_{2n}\\
&=2(F_{2n+1}+F_{2n}) =2F_{2n+2}.
\end{align*}
Then we have shown $(n+1)(3/2,1/2)=(L_{2n+2}/2,F_{2n+2}/2)$, completing the 
induction.
\end{proof}

The next result connects the order of $\alpha \in G(\F_{p})$ with
$Z(p)$.

\begin{lem}
\label{Zplem}
Let $p \ne 2, 5$ be prime. Then
\[
  Z(p) = \begin{cases}
    2 |\alpha| & \text{ if } |\alpha| \text{ is odd, }\\
    \frac{1}{2} |\alpha| & \text{ if } |\alpha| \equiv 2 \pmod{4}, \text{ and}\\
    |\alpha| & \text{ if } |\alpha| \equiv 0 \pmod{4}.
\end{cases}
\]
\end{lem}
\begin{proof}
We need a few identities involving the Lucas and Fibonacci sequences:
\begin{align*}
L^2_n-5F^2_n=4(-1)^n,\\
L^2_n=L_{2n}+2(-1)^n,\\
F_{2n}=F_nL_n.
\end{align*}
First we wish to show $p \mid F_n$ if and only if $p \mid F_{2n}$ and
$L_n^2 \equiv 4(-1)^n \pmod{p}$.  Suppose $p \mid F_n$.  Clearly $p
\mid F_nL_n=F_{2n}$ and $p \mid 5F_n=L_n^2-4(-1)^n$.  In the other
direction, if $p \mid F_{2n}$ either $p \mid F_n$ or $p \mid L_n$.  From
$L_n^2 \equiv 4(-1)^n \pmod{p}$, we can conclude that $p \nmid L_n$.
Therefore $p \mid F_n$.  

Since $L_n^2 \equiv 4(-1)^n \pmod{p}$
implies $L_{2n}+2(-1)^n \equiv 4(-1)^n \pmod{p}$, we have $L_{2n}
\equiv 2(-1)^n \pmod{p}$. Additionally as $n\alpha =
(L_{2n}/2,F_{2n}/2)$, it is clear that $p \mid F_{2n}$ and
$L_{2n}\equiv 2(-1)^n \pmod{p}$ if and only if $n\alpha \equiv
((-1)^n,0)\pmod{p}$.  Therefore we conclude $p \mid F_n$ if and only
if $n\alpha\equiv ((-1)^n,0) \pmod{p}$.

It follows that $Z(p)$ is the smallest positive integer $n$
for which $n \alpha \equiv ((-1)^{n},0) \pmod{p}$.
Since $G(\F_{p})$ is cyclic, $(-1,0)$ is the unique element in
$G(\F_{p})$ of order $2$.  If $\abs{\alpha}$ is odd, then $(-1,0)
\not\in \langle \alpha \rangle$. Therefore $n$ is even,
$\abs{\alpha}\mid n$, and $Z(p)=2\abs{\alpha}$.  If $\abs{\alpha}$
is even, then $(-1,0) \in \langle
\alpha \rangle$ and $(\abs{\alpha}/2)\alpha$ has order $2$.
When $\abs{\alpha} \equiv 2 \pmod{4}$, 
$(\abs{\alpha}/2)\alpha \equiv (-1,0) \equiv
(-1^{\abs{\alpha}/2},0) \pmod{p}$ and $Z(p)=\abs{\alpha}/2$.  When
$\abs{\alpha} \equiv 0 \pmod{4}$, then $(\abs{\alpha}/2)\alpha \equiv
(-1,0) \not\equiv (-1^{\abs{\alpha}/2},0) \pmod{p}$.  However
$\abs{\alpha} \equiv (1,0) \equiv ((-1)^{\abs{\alpha}},0) \pmod{p}$ and
$Z(p)=\abs{\alpha}$.
\end{proof}

The next lemma relates the order of $\alpha$ with the largest $m$
for which $\alpha$ has an $\ell^{m}$-th preimage. We will be able to detect
the number of preimages using Galois theoretic data, and hence determine 
the order.

\begin{lem}
\label{orderpreimage}
Suppose that $\ell$ and $p$ are prime numbers and $\alpha \in G(\F_{p})$.
If $\ell \nmid |\alpha|$, then there are infinitely many preimages
of $\alpha$ under multiplication by $\ell$. Suppose that
$\ell$ divides $|\alpha|$ and there is an $\ell^{m}$-th preimage
of $\alpha$ in $G(\F_{p})$, but no $\ell^{m+1}$-th preimage. Then
\[
  \ord_{\ell}(|\alpha|) = \ord_{\ell}(|G(\F_{p})|) - m.
\]
\end{lem}
\begin{proof}
Let $\alpha \in G(\F_{p})$, $H=\langle \alpha \rangle$ and $\phi : H \to H$
be given by $\phi(x) = \ell x$. We have that $\phi$ is an automorphism
of $H$ if and only if $\gcd(\ell,|H|) = 1$ and this implies that $\alpha$
has an $\ell^{m}$th preimage for all $m$ if and only 
if $\ord_{\ell}(|\alpha|) = 0$.

On the other hand, suppose that $\ell$ divides $|\alpha|$ and
write $\alpha = s \gamma$, where $\gamma$ is a generator of the cyclic
group $G(\F_{p})$. Write $s = \gcd(s,|G(\F_{p})|) s'$ and note that
$s' \gamma$ is also a generator of $G(\F_{p})$. By replacing
$\gamma$ with $s' \gamma$, we may assume that $s$ divides $|G(\F_{p})|$.

Let $m = \ord_{\ell}(s)$. Then $(s/\ell^{m}) \gamma$ is an $\ell^{m}$th
preimage of $\alpha$. It is easy to see, however that if an
$\ell^{m+1}$st preimage $\beta_{m+1}$ of $\alpha$ existed,
then its order would be $\ell^{m+1} |\alpha|$, and this does not divide
$|G(\F_{p})|$. Finally, the order of $s \gamma$ is
$|\alpha| = \frac{|G(\F_{p})|}{s}$ and this gives
\[
  \ord_{\ell}(|\alpha|) = \ord_{\ell}(|G(\F_{p})|) - m,
\]
as desired.
\end{proof}

\section{Galois theory}
\label{galois}

For a positive integer $m$, let $\gamma_{m} =
\sqrt[m]{\frac{3+\sqrt{5}}{2}}$.  The isomorphism in
Lemma~\ref{twistiso} shows that if $m$ is a positive integer, there
are precisely $m$ elements $P_{m,r} \in G(\C)$ with $m P_{m,r} =
\alpha$. These are given by
\[
  P_{m,r} = \left( \frac{\zeta_{m}^{r} \gamma_{m} + \zeta_{m}^{-r} \gamma_{m}^{-1}}{2}, \frac{\zeta_{m}^{r} \gamma_{m} - \zeta_{m}^{-r} \gamma_{m}^{-1}}{2 \sqrt{5}} \right), \quad 0 \leq r \leq m-1.
\]
Note that $m (P_{m,r} - P_{m,s}) = \alpha - \alpha = 0$ and so
$P_{m,r} - P_{m,s}$ is a point in $G(\C)$ with order dividing $m$.

To study the frequency with which $m$ divides
$|\alpha|$ in $G(\F_{p})$, we need to know how often the points
$P_{m,r}$ ``live in $G(\F_{p})$.'' More precisely, this means
that if $K_{m}$ is the field obtained by adjoining the $x$ and
$y$-coordinates of all the $P_{m,r}$ to $\Q$, we need to determine how
often the rational prime $p$ is contained in a prime ideal
$\mathfrak{p} \subset \mathcal{O}_{K_{m}}$ so that
\[
  \artin{K_{m}/\Q}{\mathfrak{p}}(P_{m,r}) = P_{m,r}.
\]
If the above equation is true, this implies that the point $P_{m,r}$ (which for 
all but finitely many primes $p$ can be thought of as an element of the finite 
field $\mathcal{O}_{K_{m}}/\mathfrak{p}$) is fixed by a generator of
the Galois group of $\mathcal{O}_{K_{m}}/\mathfrak{p}$ over $\F_{p}$, i.e.,
there is an element $\beta \in G(\F_{p})$ so that $m \beta = \alpha$.
We must first understand $\Gal(K_{m}/\Q)$ and
its action on the points $\{ P_{m,r} \}$. For simplicity, let
\[
  P_{m} := P_{m,0} = \left(\frac{\gamma_{m} + \gamma_{m}^{-1}}{2},
  \frac{\gamma_{m} - \gamma_{m}^{-1}}{2 \sqrt{5}}\right), 
  Q_{m} := P_{m,1} - P_{m,0} = \left(\frac{\zeta_{m} + \zeta_{m}^{-1}}{2},
  \frac{\zeta_{m} - \zeta_{m}^{-1}}{2 \sqrt{5}}\right).
\]

Let $I(m) = \{ ax + b : a \in (\Z/m\Z)^{\times}, b \in (\Z/m\Z) \}$ denote
the affine group over $\Z/m\Z$. The group law on $I(m)$ is composition.

\begin{lem}
\label{homdef}
The extension $K_{m}/\Q$ is Galois. For each $\sigma \in \Gal(K_{m}/\Q)$,
there are elements $a_{\sigma} \in (\Z/m\Z)^{\times}$ and 
$b_{\sigma} \in (\Z/m\Z)$ so that $\sigma(P_{m} + r Q_{m}) = P_{m} + (a_{\sigma} r + b_{\sigma}) Q_{m}$ for $0 \leq r \leq m-1$.
The map
\[
  \rho : \Gal(K_{m}/\Q) \to I(m)
\]
given by $\rho(\sigma) = a_{\sigma} x + b_{\sigma}$ is an injective homomorphism.
\end{lem}
\begin{proof}
  The field $K_m$ lies in the Galois extension
  $\Q(\gamma_m,\zeta_m,\sqrt{5})$ over $\Q$.  Then $K_m$ is separable
  as it is an intermediate field of a separable extension, it is a
  splitting field as we are adjoining all conjugates of the
  coordinates $P_{m}$ and $Q_{m}$.  We conclude that the extension
  $K_m/\Q$ is Galois.  For $\sigma \in \Gal(K_{m}/\Q)$, $\sigma
  (mP_{m,r})=m\sigma(P_{m,r})=\alpha$.  Therefore $\Gal(K_{m}/\Q)$
  must take a preimage of $\alpha$ to another preimage of $\alpha$.
  Hence $\sigma(P_{m,0})=P_{m,b_{\sigma}}=P_{m,0}+b_{\sigma}Q_{m}$.
  Additionally $\sigma(mQ_{m})=m\sigma(Q_{m})=0$.  So $\Gal(K_{m}/\Q)$
  must take an element of order $m$ to another element of order $m$.
  Hence $\sigma(Q_{m})=a_{\sigma}Q_{m}$.  This yields
  $\sigma(P_{m}+rQ_{m})=\sigma(P_{m})+r\sigma(Q_{m})=P_m+(a_{\sigma}
  r+ b_{\sigma})Q_{m}$. From the definition of $\rho$, it is easy to see
that $\rho$ is a homomorphism. The kernel of $\rho$ is the set $\{ \sigma
  \in \Gal(K_{m}/\Q) : \rho(\sigma)=x\}$. If $\sigma \in \ker \rho$,
then $\sigma(P_{m} + rQ_{m}) = P_{m} + rQ_{m}$ for all $r$. This implies
that $P_{m}$ and $Q_{m}$ are both fixed by $\sigma$. 
Since the coordinates of $P_{m}$ and $Q_{m}$ generate $K_{m}$ over $\Q$,
it follows that $\sigma$ fixes $K_{m}$ and so $\ker \rho = 1$.
\end{proof}

The Chinese remainder theorem allows one to see that if the prime factorization
of $m = \prod_{i=1}^{r} \ell_{i}^{s_{i}}$, then
\[
  I(m) \cong \prod_{i=1}^{r} I(\ell_{i}^{s_{i}}).
\]
We will therefore begin by studying the images of $\Gal(K_{m}/\Q)
\to I(m)$ in the case that $m$ is a prime power. This case was studied
by Jones and the second author in \cite{JonesRouse}. In particular,
Proposition 4.1 and Theorem 4.2 (with the special case $d = 5$) handle the 
case that $m = \ell^{k}$ is a prime. Specifically, if $F$ is a number field, 
the map
\[
  \rho : \Gal(\langle K_{m}, F \rangle/F) \to I(\ell^{k})
\]
is surjective for all $k$ if and only if
\begin{enumerate}
\item there is no point $\beta \in G(F)$ with $\ell \beta = \alpha$, and
\item $|F(\zeta_{\ell^{3}} + \zeta_{\ell^{3}}^{-1}) : F| = \frac{\ell^{2} (\ell-1)}{2}$, and
\item if $\ell = 2$, $-2$ and $-10$ are not squares in $F$ and
if $L$ is the field obtained by adjoining to $F$ the coordinates of $Q_{8}$,
then there is no point $\beta \in G(L)$ with $2 \beta = \alpha$.
\end{enumerate}

Using this result, we deduce the following result concerning the
surjective of the map $\rho : \Gal(K_{m}/\Q) \to I(m)$.

\begin{lem}
\label{surjlem}
Let $\ell$ be a prime and $k \geq 1$. Then the map 
$\rho : \Gal(K_{\ell^{k}}/\Q) \to I(\ell^{k})$ is surjective.
\end{lem}
\begin{proof}
  If there is a $\beta \in G(\Q(\sqrt{5}))$ with $\ell \beta = \alpha$,
then the isomorphism $G(\Q(\sqrt{5})) \cong H(\Q(\sqrt{5}))$ from 
Lemma~\ref{twistiso} sends $\beta$ to $\sqrt[\ell]{(3 + \sqrt{5})/2}$.
It is well-known that $\mathcal{O}_{\Q(\sqrt{5})} = \Z\left[\frac{1 + \sqrt{5}}{2}\right]$ and $\Z\left[\frac{1 + \sqrt{5}}{2}\right]^{\times}$ is generated
by $-1$ and $\frac{1+\sqrt{5}}{2}$. Since $\frac{3 + \sqrt{5}}{2} = \left(\frac{1 + \sqrt{5}}{2}\right)^{2}$, it follows that $\ell = 2$. However,
the $2$nd preimages of $\alpha$ are $\left( \frac{\pm \sqrt{5}}{2},
\frac{\pm \sqrt{5}}{10} \right)$, which are not in
$G(\Q)$. This establishes condition $(1)$. It is well-known that
  $|\Q(\zeta_{\ell^{3}} + \zeta_{\ell^{3}}^{-1}) : \Q| =
  \frac{\ell^{2} (\ell-1)}{2}$, and so condition $(2)$ holds. For condition (3),
a simple check shows that $L = \Q(\sqrt{2}, \sqrt{-5})$ and so $\sqrt{5}
\not\in L$ and thus there is no $\beta \in G(L)$ with $2 \beta = \alpha$.
\end{proof}

The next step is extending the result of the above lemma to
prove that the map $\Gal(K_{m}/\Q) \to I(m)$ is surjective (unless
$10 | m$).

\begin{lem}
\label{minsubfields}
If $\ell > 2$ is prime, then every minimal subfield of $K_{\ell^{k}}/\Q$
is either $\Q\left(P_{\ell}\right)$ (or some conjugate),
or is contained in $\Q(\zeta_{\ell^{2}})$. If $\ell = 2$, then every minimal 
subfield of $K_{2^{k}}/\Q$ is contained in $\Q(\sqrt{5}, i, \sqrt{2})$.
\end{lem} 
\begin{proof}
Lemmas~\ref{homdef} and \ref{surjlem} translate this into a group theory
problem. Let $N = \{ ax+b \in I(\ell^{k}) : a \equiv 1 \pmod{\ell}\}$.
This is a normal subgroup of $I(\ell^{k})$. If $G$ is a group, let $\Phi(G)$
denote the Frattini subgroup of $G$, the intersection of the maximal subgroups
of $G$. Since $N \unlhd G$, $\Phi(N) \subseteq \Phi(G)$. Since the
order of $N$ is a power of $\ell$, every maximal subgroup has index $\ell$
and so if $n \in N$, then $n^{\ell} \in \Phi(N)$. This implies that
\[
  \{ ax+b \in I(\ell^{k}) : a \equiv 1 \pmod{\ell^{2}}, b \equiv 0 \pmod{\ell}\} \subseteq \Phi(N) \subseteq \Phi(G).
\]
and in particular, that the kernel of the map from $I(\ell^{k})
\to I(\ell^{2})$ is contained in $\Phi(G)$. Since $I(\ell^{2})$ is solvable,
every maximal subgroup has prime power index. Using the surjective homomorphism
$\phi : I(\ell^{2}) \to (\Z/\ell^{2} \Z)^{\times}$, we can identify the maximal
subgroups of index less than $\ell$ and also show there are none of index 
$> \ell$. If $M \subseteq I(\ell^{2})$ is a maximal subgroup of index $\ell$
then either $x+1 \in M$, in which case $M = \{ ax+b : a^{\ell-1} \equiv 1 \pmod{\ell^{2}} \}$, or $x+1 \not\in M$. Note that if $f_{1} = ax + b$
and $f_{2} = cx+d$, then $f_{1} \circ f_{2} \circ f_{1}^{-1} \circ f_{2}^{-1}
= x + ad - bc + b - d$. Hence, if $x+1 \not\in M$, then $M$ is abelian.
This implies that if $ax+b$ and $cx+d$ are in $M$ then $(a-1)d = (c-1) b$.
If $e$ is the common value of $\frac{b}{a-1}$ for $a \ne 1$, then we have
\[
  M = \{ a (x-e) + e : a \in (\Z/\ell^{2} \Z)^{\times} \}
\]
for some $e$ with $0 \leq e \leq \ell^{2} - 1$. All of these subgroups
are conjugates of $\{ ax : a \in (\Z/\ell^{2} \Z)^{\times}
\}$. Translating back to fields, we obtain the desired result when
$\ell > 2$.

When $\ell = 2$  we have $\Phi(I(2^{k})) = \{ ax + b : a \equiv 1 \pmod{8},
b \equiv 0 \pmod{2} \}$ provided $k \geq 3$, and a straightforward calculation
shows that the field corresponding to $\Phi(I(2^{k}))$ is
$\Q(\sqrt{5}, i, \sqrt{2})$.
\end{proof}

Suppose that $\ell_{1}, \ell_{2}, \ldots, \ell_{n}$ are prime divisors of $m$.
To prove that $\Gal(K_{m}/\Q) \to I(m)$ is surjective, we will want
to show that $K_{\ell_{n}^{s_{n}}} \cap \langle K_{\ell_{1}^{s_{1}}},
K_{\ell_{2}^{s_{2}}}, \ldots, K_{\ell_{n-1}}^{s_{n-1}} \rangle = \Q$. We will
prove this using ramification properties of these fields.

\begin{lem}
\label{ramify}
Suppose that $\ell$ is prime. Then $K_{\ell^{k}}/\Q$ is ramified
only at $5$ and $\ell$. If $\ell \ne 2$, then every minimal subfield
of $K_{\ell^{k}}/\Q$ is ramified at $\ell$. 
\end{lem}
\begin{proof}
  From the formulas for the $P_{\ell^{k},r}$ it is clear that
  $K_{\ell^{k}}$ is contained in the splitting field of $x^{\ell^{k}}
  - \frac{3 + \sqrt{5}}{2}$ over $\Q(\sqrt{5})$. The discriminant
  $\Delta_{L/\Q}$ of $L = \Q(\gamma_{\ell^{k}})$ is
\[
  \Delta_{L/\Q} = \Delta_{\Q(\sqrt{5})/\Q}^{[L : \Q(\sqrt{5})]}
  \cdot N_{L/\Q(\sqrt{5})}(\Delta_{L/\Q(\sqrt{5})}).
\]
Since the discriminant of $x^{\ell^{k}} - \frac{3 + \sqrt{5}}{2}$
is a power of $\ell$ times a unit, it follows that $\Delta_{L/\Q}$ is a power
of $5$ times a power of $\ell$. If $L_{1}$ and $L_{2}$ are two extensions of
$\Q$ ramified only at primes in a set $S$, then $\langle L_{1}, L_{2} \rangle/\Q$
is ramified only at primes in $S$ (see Theorem 4.67 of \cite{Mollin}). It follows from this that the splitting field of $x^{\ell^{k}} - \frac{3 + \sqrt{5}}{2}$
is ramified only at $5$ and $\ell$ and hence $K_{\ell^{k}}/\Q$ is too.

To prove the second claim it is enough, by Lemma~\ref{minsubfields},
to prove that $\Q(P_{\ell})$ is ramified at $\ell$. From the
surjectivity of $\rho$ proven in Lemma~\ref{surjlem}, it follows that
$\Gal(K_{\ell}/\Q)$ acts transitively on the $P_{\ell,r}$. Therefore,
the Galois closure of $\Q(P_{\ell})$ over $\Q$ is $K_{\ell}$. Since
$\Q(\zeta_{\ell} + \zeta_{\ell}^{-1})$ is ramified at $\ell$ and is
contained in $K_{\ell}$, it follows that $K_{\ell}/\Q$ is ramified at $\ell$
and this implies that $\Q(P_{\ell})$ is ramified at $\ell$ as well.
\end{proof}

\begin{lem}
\label{surj}
Let $m$ be an arbitrary positive integer. The map
$\rho : \Gal(K_{m}/\Q) \to I(m)$ is surjective if $10 \nmid m$.
If $10 | m$, then the image of $\rho$ is
\[
  \{ ax + b : b \text{ is even } \iff a \equiv 1 \text{ or } 4 \pmod{5} \}.
\]
\end{lem}
\begin{proof}
First note that if $\gcd(m_{1},m_{2}) = 1$, then
$\langle K_{m_{1}}, K_{m_{2}} \rangle = K_{m_{1} m_{2}}$. It is clear that
$K_{m_{1}}, K_{m_{2}} \subseteq K_{m_{1} m_{2}}$. For the reverse direction
note that if $m_{1} x + m_{2} y = 1$ then
\[
  m_{1} m_{2} (y P_{m_{1}} + x P_{m_{2}})
  = m_{2} y \alpha + m_{1} x \alpha = \alpha.
\]
Thus $y P_{m_{1}} + x P_{m_{2}}$ is preimage of $\alpha$ under multiplication
by $m_{1} m_{2}$. Further, since $G(\langle K_{m_{1}}, K_{m_{2}} \rangle)$ 
contains elements of order $m_{1}$ and $m_{2}$, it must contain $Q_{m_{1} m_{2}}$ 
and so $K_{m_{1} m_{2}} = \langle K_{m_{1}}, K_{m_{2}} \rangle$.

Suppose first that $10 \nmid m$. To prove that $\rho$ is surjective,
we will prove (by induction on the number of distinct prime factors of
$m$) that $|\Gal(K_{m}/\Q)| = |I(m)|$. The base case ($n = 1$) is
handled by Lemma~\ref{surjlem}.

Suppose that $m = \prod_{i=1}^{n} \ell_{i}^{s_{i}}$ is the prime
factorization of $m$. Since $K_{m/\ell_{n}^{s_{n}}} = \langle
K_{\ell_{1}^{s_{1}}}, K_{\ell_{2}^{s_{2}}}, \ldots,
K_{\ell_{n-1}^{s_{n-1}}} \rangle$ we have that
$K_{m/\ell_{n}^{s_{n}}}$ is ramified only at $5$ and $\ell_{1},
\ldots, \ell_{n-1}$, while $K_{\ell_{n}^{s_{n}}}$ is ramified only at
$5$ and $\ell_{n}$.  Moreover, by Lemma~\ref{ramify} every minimal
subfield of $K_{\ell_{n}^{s_{n}}}$ is ramified at $\ell_{n}$. It
follows therefore that $K_{\ell_{n}^{s_{n}}} \cap
K_{m/\ell_{n}^{s_{n}}}$ is unramified everywhere, and since $\Q$ has
no unramified extensions we have that $K_{\ell_{n}^{s_{n}}} \cap
K_{m/\ell_{n}^{s_{n}}} = \Q$. From Theorem~\ref{linearlydisjoint} we obtain that
\[
  |\langle K_{\ell_{n}^{s_{n}}}, K_{m/\ell_{n}^{s_{n}}} \rangle : \Q|
  = |K_{\ell_{n}^{s_{n}}} : \Q| \cdot |K_{m/\ell_{n}^{s_{n}}} : \Q|
  = |I(\ell_{n}^{s_{n}})| \cdot |I(m/\ell_{n}^{s_{n}})| = |I(m)|,
\]
which proves the desired claim.

A similar argument shows that if $10 | m$ then $|K_{m} : \Q| =
|K_{2^{s_{1}} 5^{s_{2}}} : \Q| \cdot |K_{m/(2^{s_{1}} 5^{s_{2}})} : \Q|$.
To determine $|K_{2^{s_{1}} 5^{s_{2}}} : \Q|$ we will show that
$K_{2^{s_{1}}} \cap K_{5^{s_{2}}} = \Q(\sqrt{5})$ by determining the
subfields of $K_{2^{s_{1}}}$ that are ramified only at $5$. We have
that $\Q(\sqrt{5}) \subseteq K_{2^{s_{1}}}$ and the subgroup of
$I(2^{s_{1}})$ corresponding to $\Q(\sqrt{5})$ is
\[
  H = \{ ax + b : b \equiv 0 \pmod{2} \}.
\]
Since $H$ is a $2$-group, $\Phi(H) = H' H^{2} = \{ ax + b
: a \equiv 1 \pmod{8}, b \equiv 0 \pmod{4} \}$. The field corresponding
to this subgroup is $L = \Q\left(\sqrt{\frac{1+\sqrt{5}}{2}}, i, \sqrt{2}\right)$ and it
is straightforward to see that the maximal subextension of $L$ ramified
only at $5$ is $\Q(\sqrt{5})$. It follows that $K_{2^{s_{1}}} \cap
K_{5^{s_{2}}} = \Q(\sqrt{5})$ and so when $10 | m$, $|K_{m} : \Q| =
\frac{1}{2} |I(m)|$. Finally, since $P_{2} = \left(\frac{\sqrt{5}}{2},
\frac{\sqrt{5}}{10} \right)$ we have that $\sigma \in \Gal(K_{m}/\Q)$
fixes $\sqrt{5}$ if and only if it fixes $P_{2}$ and this occurs if and only if
$\rho(\sigma) = ax + b$ where $b \equiv 0 \pmod{2}$ and $\legen{a}{5} = 1$.
This yields the desired result.
\end{proof}

\begin{rem}
The method from the previous sections is very general and can be used
to establish the surjectivity of Galois representations attached to
arbitrary one-dimensional tori.
\end{rem}

\section{Density computations}
\label{density}

In this section, we will translate conditions on when preimages of
$\alpha$ exist in $G(\F_{p})$ into statements about the Frobenius
conjugacy class $\artin{K_{m}/\Q}{p}$. We will then count the sizes of
these classes and use this to prove Theorem~\ref{main}. We will
start by focusing on the prime power case.

As in the previous section, let $\ell$ be a prime number, $k \geq 1$ and 
$K_{\ell^{k}}$ be the field
obtained by adjoining all the $x$ and $y$-coordinates of $P_{\ell^{k},r}$
to $\Q$. Let $\rho : \Gal(K_{\ell^{k}}/\Q) \to I(\ell^{k})$ be the homomorphism
defined in Lemma~\ref{homdef}. We define
\[
  \mathcal{C}_{k,n,\ell} = \{ \sigma \in \Gal(K_{\ell^{k}}/\Q) :
\sigma(\ell^{k-n} P_{\ell^{k},r}) = \ell^{k-n} P_{\ell^{k},r} \text{ for some
 } r \text{ with } 0 \leq r \leq \ell^{k} - 1 \}.
\]
We will make the convention that if $n < 0$, then
$\mathcal{C}_{k,n,\ell}$ is empty. Since $P_{\ell^{k},r}$ is a
$\ell^{k}$th preimage of $\alpha$, $\ell^{k-n} P_{\ell^{k},r}$ is an
$\ell^{n}$th preimage of $\alpha$.

\begin{lem}
\label{artinpreimage}
Let $p \ne 2, 5, \ell$ be a prime number. Then,
there is an $n$th preimage of $\alpha$ in $G(\F_{p})$ if and only if
$\artin{K_{\ell^{k}}/\Q}{p} \subseteq \mathcal{C}_{k,n,\ell}$.
\end{lem}
\begin{proof}
  First suppose that $\artin{K_{\ell^{k}}/\Q}{p} \subseteq
  \mathcal{C}_{k,n,\ell}$ and let $\mathfrak{p}$ be a prime ideal
  above $p$ in $\mathcal{O}_{K_{\ell^{k}}}$. Let $\sigma =
  \artin{K_{\ell^{k}}/\Q}{\mathfrak{p}}$ fix $\ell^{k-n}
  P_{\ell^{k},r}$. We may consider $\ell^{k-n} P_{\ell^{k},r}$ as an
  element of $G(\mathcal{O}_{K_{\ell^{k}}}/\mathfrak{p})$ (note that
  $2$ and $5$ are the only primes that divide the denominators of
  coordinates of preimages of $\alpha$). Since $\sigma$ fixes
  $\ell^{k-n} P_{\ell^{k},r}$ and $\sigma$ acts as the Frobenius
  automorphism on $\mathcal{O}_{K_{\ell^{k}}}/\mathfrak{p}$, it
  follows that $\ell^{k-n} P_{\ell^{k},r} \in G(\F_{p})$, as desired.

To show the reverse implication we will first show that for any
prime $\mathfrak{p}$ above $p$, the
reduction mod $\mathfrak{p}$ map on $\ell^{n}$th preimages of $\alpha$
is injective. Any $n$th preimage of $\alpha$ has the form
$P_{\ell^{n},r} = P_{\ell^{n}} + r Q_{\ell^{n}}$. If $P_{\ell^{n}} + r_{1}
Q_{\ell^{n}} \equiv P_{\ell^{n}} + r_{2} Q_{\ell^{n}} \pmod{\mathfrak{p}}$
with $r_{1} \not\equiv r_{2} \pmod{\ell^{n}}$, then
$(r_{1} - r_{2}) Q_{\ell^{n}}$ is congruent to the identity mod $\mathfrak{p}$.
Every element of order $\ell$ is a multiple of $(r_{1} - r_{2}) Q_{\ell^{n}}$
and this implies that the $y$-coordinate of $Q_{\ell}$ is $\equiv 0 \pmod{\mathfrak{p}}$ and this implies that
\[
  N_{\Q(\zeta_{\ell}, \sqrt{5})/\Q}\left(\frac{\zeta_{\ell} - \zeta_{\ell}^{-1}}{2 \sqrt{5}}\right) \equiv 0 \pmod{p}.
\]
This is a contradiction because $N_{\Q(\zeta_{\ell})/\Q}(\zeta_{\ell} - \zeta_{\ell}^{-1}) = \ell$. Hence the reduction mod $\mathfrak{p}$ map is injective
on $\ell^{n}$th preimages. 

Finally, suppose there is an $\ell^{n}$th preimage of $\alpha$ in $G(\F_{p})$.
For a prime ideal $\mathfrak{p}$ above $p$, we consider
$\F_{p} \subseteq \mathcal{O}_{K_{\ell^{k}}}/\mathfrak{p}$. Since $\Gal(K_{\ell^{k}}/\Q)$ acts on the $\ell^{n}$th preimages of $\alpha$ in $K_{\ell^{k}}$, and
the reduction map is injective, this implies that $\artin{K_{\ell^{k}}/\Q}{\mathfrak{p}} \in \mathcal{C}_{k,n,\ell}$, as desired. 
\end{proof}

The previous lemma allows us to determine, based on $\artin{K/\Q}{p}$, when
preimages exist. The next allows us to determine the group order mod
$\ell^{k}$.
\begin{lem}
\label{artingrouporder}
If $p$ is a prime number with $p \ne 2, 5, \ell$. If
\[
  ax+b \in \rho\left(\artin{K_{\ell^{k}}/\Q}{p}\right) \subseteq I(\ell^{k}),
\]
and $a \not\equiv 1 \pmod{\ell^{k}}$, then
$\ord_{\ell}(|G(\F_{p})|) = \ord_{\ell}(a-1)$.
\end{lem}
\begin{proof}
Let $\mathfrak{p}$ be a prime above $p$ in $\mathcal{O}_{K_{\ell^{k}}}$
and $\sigma = \artin{K_{\ell^{k}}/\Q}{\mathfrak{p}}$. Note
that $\sigma(\zeta_{\ell^{k}}) = \zeta_{\ell^{k}}^{p}$ and define
$\epsilon \in \{ \pm 1 \}$ by $\sigma(\sqrt{5}) = \epsilon \sqrt{5}$. If
$\rho(\sigma) = ax + b$, then
\[
  \left(\frac{\zeta_{\ell^{k}}^{p} + \zeta_{\ell^{k}}^{-p}}{2},
  \frac{\zeta_{\ell^{k}}^{p} - \zeta_{\ell^{k}}^{-p}}{2 \epsilon \sqrt{5}}\right) = 
  \sigma(Q_{\ell^{k}}) = a Q_{\ell^{k}}
  = \left(\frac{\zeta_{\ell^{k}}^{a} + \zeta_{\ell^{k}}^{-a}}{2}, \frac{\zeta_{\ell^{k}}^{a} - \zeta_{\ell^{k}}^{-a}}{2 \sqrt{5}}\right).
\]
Comparing the $x$-coordinates, we obtain that $a \equiv \pm p
\pmod{\ell^{k}}$.  Using this fact and comparing the $y$-coordinates
gives that $a \equiv \epsilon p \pmod{\ell^{k}}$. Also, $\epsilon = 1$
if and only if $\sqrt{5} \in \mathcal{O}_{K_{\ell^{k}}}/\mathfrak{p}$
and so $\epsilon = \legen{p}{5}$. Finally, we use
Lemma~\ref{grouporder} to conclude that
\[
  a - 1 \equiv \begin{cases}
    p-1 \pmod{\ell^{k}} & \text{ if } p \equiv 1 \text{ or } 4 \pmod{5}\\
   -p-1 \pmod{\ell^{k}} & \text{ if } p \equiv 2 \text{ or } 3 \pmod{5}.
\end{cases}
\] 
This yields the desired result.
\end{proof}

Now, for $1 \leq t < k$, define
\[
  \mathcal{D}_{k,t,\ell}
  = \{ \sigma \in \Gal(K_{\ell^{k}}/\Q) : a_{\sigma} \not\equiv
    1 \pmod{\ell^{k}} \text{ and if }
    n = \ord_{\ell}(a_{\sigma} - 1), \text{ then }
    \sigma \not\in \mathcal{C}_{k,n-t+1,\ell} \}.
\]
Here $a_{\sigma}$ denotes the coefficient of $x$ in
$\rho(\sigma) = a_{\sigma} x + b_{\sigma} \in I(\ell^{k})$.
Combining Lemma~\ref{orderpreimage}, Lemma~\ref{artinpreimage},
and Lemma~\ref{artingrouporder}, we see that if $p$ is a prime
with $p \not\equiv \pm 1 \pmod{\ell^{k}}$, then
\[
  \ell^{t} \text{ divides } |\alpha| \text{ if and only if }
  \artin{K_{\ell^{k}}/\Q}{p} \subseteq \mathcal{D}_{k,t,\ell}.
\]

\begin{lem}
\label{countprime}
Assume the notation above. For $1 \leq t < k$, we have
\[
  \frac{|\mathcal{D}_{k,t,\ell}|}{|\Gal(K_{\ell^{k}}/\Q)|} = 
  \frac{\ell^{2-t} - \ell^{2-k} - \ell^{1-k} + \ell^{1-2k+t}}{\ell^{2} - 1}. 
\]
\end{lem}
\begin{proof}
  Noting that $P_{m,r} = P_{m} + r Q_{m}$ we see that $\sigma$
  fixes $\ell^{k-n+t+1} P_{\ell^{k},r}$ if and only if $(a_{\sigma}-1)
  r + b_{\sigma} \equiv 0 \pmod{\ell^{n-t+1}}$. It is easy to see that
  there are no solutions $r$ to this congruence if and only if
  $\ell^{n-t+1} \nmid b$. Thus, $|\mathcal{D}_{k,t,\ell}| = \{ (a,b)
  \in I(\ell^{k}) : \ord_{\ell}(a - 1) = n \text{ and } \ell^{n-t+1}
  \nmid b \}$. We find that there are $\ell^{2k -2n + t - 2} (\ell -
  1) (\ell^{n-t+1} - 1)$ elements of $I(\ell^{k})$ satisfying these
  properties if $n \geq t$, and $0$ if $n < t$.  Summing from $n = t$
  to $k-1$ we get
\[
  |\mathcal{D}_{k,t,\ell}| = \frac{\ell^{2k-t+1} - \ell^{k+1} + \ell^{t} - \ell^{k}}{\ell + 1}
\]
and dividing by $|I(\ell^{k})| = \ell^{2k-1} (\ell - 1)$ gives the desired 
result.
\end{proof}    

We are now able to prove Theorem~\ref{BAconj2} using similar reasoning.

\begin{proof}[Proof of Theorem~\ref{BAconj2}]
  Fix a prime $q \ne 2$ and integers $i$ and $j$ with $i \geq j \geq
  0$. By Lemma~\ref{Zplem}, we have $\ord_{q}(Z(p)) =
  \ord_{q}(|\alpha|)$ for $\alpha = (3/2,1/2) \in G(\F_{p})$. Let $k =
  i+1$ and $p \ne 2, 5, q$ be a prime.  As in the proof of
Lemma~\ref{countprime}, we have $\ord_{q}(p -
  \epsilon_{p}) = i$ and $\ord_{q}(Z(p)) = j$ if and only if when
$\sigma \in \artin{K_{q^{k}}/\Q}{p}$, then
\[
  \ord_{q}(a_{\sigma} - 1) = i \text{ and } \ord_{q}(b_{\sigma})
  = i-j, 
\]
(with the exception that when $j = 0$, there is a $q^{n}$th preimage
of $\alpha$ for all $n$ and so $\ord_{q}(b_{\sigma}) \geq i$).
When $i$ and $j$ are both positive, there are $(q-1)^{2} q^{2k-2i+j-2}$
pairs of $(a_{\sigma},b_{\sigma})$. When $i = 0$, there are $(q-2) q^{2k-2}$
pairs, and when $i \geq 1$ and $j = 0$, there are $(q-1) q^{2k-2i-1}$ pairs.
Applying the Chebotarev density theorem proves the desired result.
\end{proof}

Now, we turn to the composite case. Let $m = \prod_{i=1}^{r}
\ell_{i}^{s_{i}}$.  Take $M = \prod_{i=1}^{r} \ell_{i}^{S_{i}}$ to be
a multiple of $m$. The discussion following the proof of Lemma~\ref{artingrouporder} now
implies the following statement. If $p \not\equiv \pm 1
\pmod{\ell_{i}^{S_{i}}}$ for all $i$, $1 \leq i \leq r$, then $m$
divides the order of $\alpha$ in $G(\F_{p})$ if and only if for all
$\sigma \in \legen{K_{M}/\Q}{p}$,
\[
  \sigma|_{K_{\ell_{i}^{S_{i}}}} \in \mathcal{D}_{S_{i},s_{i},\ell_{i}}
  \text{ for all } i, 1 \leq i \leq r.
\]
In the case that $m$ is coprime to $10$, it follows from Lemma~\ref{surj} that
$\Gal(K_{M}/\Q) \cong \prod_{i=1}^{r} \Gal(K_{\ell_{i}^{S_{i}}}/\Q)$,
and the fraction of elements $\sigma$ that are in
$\mathcal{D}_{S_{i},s_{i},\ell_{i}}$ for all $i$ is
\[
  \prod_{i=1}^{r} \frac{|\mathcal{D}_{S_{i},s_{i},\ell_{i}}|}{|\Gal(K_{M}/\Q)|}.
\]
In the case that $m$ is a multiple of $10$, $\Gal(K_{M}/\Q)$ is an
index 2 subgroup of the direct product and counting is more tricky.
For this reason, we now define for $t_{1} \geq 1$ and $t_{2} \geq 1$
\begin{align*}
  & \mathcal{D}_{k,t_{1},t_{2},10}
  = \left\{ \sigma \in \Gal(K_{10^{k}}/\Q) :
  a_{\sigma} \not\equiv 1 \pmod{2^{k}} \text{
  and } a_{\sigma} \not\equiv 1 \pmod{5^{k}} \text{ and if }\right.\\
  & \left. n_{1} = \ord_{2}(a_{\sigma} - 1), n_{2} = \ord_{5}(a_{\sigma} - 1),
  \text{ then } \sigma \not\in \mathcal{C}_{k,n_{1}-t_{1}+1,2}
  \text{ and } \sigma \not\in \mathcal{C}_{k,n_{2}-t_{2}+1,5} \right\}.
\end{align*}
For $t_{1} = 0$, we omit the condition $\sigma \not\in
\mathcal{C}_{k,n_{1}-t_{1}+1,2}$.

\begin{lem}
\label{count25}
Assume the notation above. For $0 \leq t_{1} < k$
and $1 \leq t_{2} < k$, we have
\begin{align*}
  \frac{|\mathcal{D}_{k,t_{1},t_{2},10}|}{|\Gal(K_{10^{k}}/\Q)|}
  &= \frac{25}{36 \cdot 2^{t_{1}} 5^{t_{2}}}
  - \frac{5}{6 \cdot 2^{t_{1}} 5^{k}} +
  \frac{5}{36 \cdot 2^{t_{1}} 5^{2k-t_{2}}} + \frac{5}{2 \cdot 10^{k}}
  - \frac{5}{12 \cdot 2^{k} 5^{2k-t_{2}}}\\ 
  &- \frac{25}{12 \cdot 5^{t_{2}} 2^{k}}
  + \frac{5}{18 \cdot 2^{2k - t_{1}} 5^{2k-t_{2}}} +
  \frac{25}{18 \cdot 2^{2k-t_{1}} 5^{t^{2}}} - \frac{5}{3 \cdot 2^{2k-t_{1}} 5^{k}}, 
\end{align*}
when $t_{1} > 0$, and
\begin{align*}
  \frac{|\mathcal{D}_{k,0,t_{2},10}|}{|\Gal(K_{10^{k}}/\Q)|}
  &= \frac{25}{9 \cdot 5^{t_{2}}} -
  \frac{25}{9 \cdot 5^{t_{2}} \cdot 4^{k}}
  + \frac{5}{9 \cdot 5^{2k-t_{2}}} - \frac{10}{3 \cdot 5^{k}}
  - \frac{5}{9 \cdot 4^{k} \cdot 5^{2k-t_{2}}}
  + \frac{10}{3 \cdot 20^{k}}.
\end{align*}
\end{lem}
\begin{proof}
Arguing as in Lemma~\ref{countprime}, $|\mathcal{D}_{k,t_{1},t_{2},10}|$ is
the number of $ax+b \in I(10^{k})$ that satisfy
\begin{align*}
  & b \text{ even } \iff a \equiv 1, 4 \pmod{5},\\
  & a \not\equiv 1 \pmod{2^{k}}, a \not\equiv 1 \pmod{5^{k}}\\
  & \ord_{2}(a-1) = n_{1}, \ord_{5}(a - 1) = n_{2},\\ 
  & 2^{n_{1} - t_{1} + 1} \nmid b, \text{ and } 5^{n_{2} - t_{2} + 1} \nmid b,
\end{align*}
provided $t_{1} > 0$. In the case that $b$ is odd, $a \equiv 2, 3
\pmod{5}$ and $t_{2} = 0$. In the case that $b$ is even and $a \equiv
1, 4 \pmod{5}$ there are $2^{k-n_{1} - 1} \cdot 4 \cdot 5^{k - n_{2} -
  1}$ choices for $a$ and $(2^{k-1} - 2^{k - n_{1} + t_{1} - 1})
(5^{k} - 5^{k - n_{1} + t_{1} - 1})$ choices for $b$. Summing over
$n_{1}$ and $n_{2}$ gives
\[
  \sum_{n_{1} = t_{1}}^{k-1} \sum_{n_{2} = t_{2}}^{k-1}
  4 \cdot 2^{k - n_{1} - 1} \cdot 5^{k - n_{2} - 1}
  \left(2^{k-1} - 2^{k - n_{1} + t_{1} - 1}\right)
  \left(5^{k} - 5^{k - n_{2} + t_{2} - 1}\right).
\]
This sum of four geometric series is easily evaluated to give the
stated answer. In the case that $t_{1} = 0$, there are again
$2^{k-n_{1}-1} \cdot 4 \cdot 5^{k-n_{2}-1}$ choices for $a$,
and $2^{k-n_{1}} \cdot (5^{k} - 5^{k-n_{2}+t_{2} - 1})$ choices for $b$. Summing
yields the stated result.
\end{proof}

We are now ready to prove the main result.
\begin{proof}[Proof of Theorem~\ref{main}]
  Let $m$ be a positive integer with $\gcd(m,10) = 1$ and fix an
  $\epsilon > 0$.  Let $m = \prod_{i=1}^{r} \ell_{i}^{s_{i}}$ be the
  prime factorization of $m$ and note that $\lim_{k \to \infty}
  \frac{|\mathcal{D}_{k,t,\ell}|}{|\Gal(K_{\ell^{k}}/\Q)|} =
  \zeta(\ell^{t})$ by Lemma~\ref{countprime}. Choose a positive real number $\eta$ small enough
  so that $1 - \frac{\epsilon}{2} \leq (1 - \eta)^{r}$ and $(1 +
  \eta)^{r} \leq 1 + \frac{\epsilon}{2}$. Now, let $M =
  \prod_{i=1}^{r} \ell_{i}^{S_{i}}$ be chosen with $S_{i}$ is
  sufficiently large that
\[
  \zeta(\ell^{s_{i}}) - \eta \leq \frac{|\mathcal{D}_{S_{i},s_{i},\ell_{i}}|}{|\Gal(K_{\ell_{i}^{S_{i}}}/\Q)|} \leq \zeta(\ell^{s_{i}}) - \frac{2}{(\ell_{i} - 1) \ell_{i}^{S_{i}}} + \eta
\]
for $1 \leq i \leq r$. Combining the Chebotarev density theorem with
the observation that $\ell_{i}^{s_{i}}$ divides the order of $\alpha \in
G(\F_{p})$ if and only if $\artin{K_{\ell^{S_{i}}}/\Q}{p} \subseteq \mathcal{D}_{S_{i},s_{i},\ell_{i}}$ (provided $p \not\equiv \pm 1 \pmod{\ell_{i}^{S_{i}}}$) we obtain
that the number of primes $p \leq x$ for which $m$ divides $|\alpha|$ satisfies
\[
  -\epsilon/2 + \prod_{i=1}^{r} \frac{|\mathcal{D}_{S_{i},s_{i},\ell_{i}}|}{\Gal(K_{\ell_{i}^{S_{i}}}/\Q)} \leq \frac{\# \{ p \leq x : m | |\alpha| \}}{\pi(x)}
  \leq \prod_{i=1}^{r} \frac{|\mathcal{D}_{S_{i},s_{i},\ell_{i}}|}{\Gal(K_{\ell_{i}^{S_{i}}}/\Q)}
  + \frac{2}{(\ell_{i} - 1) \ell_{i}^{S_{i} - 1}} + \epsilon/2
\]
provided $x$ is sufficiently large. We have that
\[
   \zeta(m) (1 - \epsilon/2) \leq \prod_{i=1}^{r} \left(\zeta(\ell_{i}^{s_{i}}) - \eta\right) \leq d \leq \prod_{i=1}^{r} \left(\zeta(\ell_{i}^{s_{i}}) + \eta\right) \leq
  \zeta(m) (1 + \epsilon/2).
\]
which implies that 
\[
  (\zeta(m) - \epsilon) \pi(x) \leq \# \{ p \leq x : m | |\alpha| \} \leq (\zeta(m) + \epsilon) \pi(x),
\]
provided $x$ is large enough. Combining this with
Lemma~\ref{Zplem} proves Theorem~\ref{main} in the case that
$\gcd(m,10) = 1$. The case that $\gcd(m,10) > 1$ is similar with two
notable differences: there is extra complexity in dealing with the
primes $2$ and $5$ using Lemma~\ref{count25}, and Lemma~\ref{Zplem}
shows that $2 | Z(p)$ if and only if $|\alpha|$ is odd.
\end{proof}

\bibliographystyle{amsplain}
\bibliography{refs}

\end{document}